\documentclass[12pt,leqno]{amsart}
\usepackage{amsmath,amsfonts,amssymb,amsthm}

\theoremstyle{plain}
\newtheorem{theorem}{Theorem}
\newtheorem{lemma}[theorem]{Lemma}
\newtheorem{proposition}[theorem]{Proposition}

\theoremstyle{definition}
\newtheorem{definition}{Definition}[section]

\theoremstyle{remark}

\begin{document}
\title[$\eta$-Ricci Soliton and gradient $\rho$-Einstein soliton]{Some results on $\eta$-Ricci Soliton and gradient $\rho$-Einstein soliton in a complete Riemannian manifold}
\author[A. A. Shaikh C. K. Mondal]{Absos Ali Shaikh$^{1*}$ and Chandan Kumar Mondal$^2$}

\begin{abstract}
The main purpose of the paper is to prove that if a compact Riemannian manifold admits a gradient $\rho$-Einstein soliton such that the gradient Einstein potential is a non-trivial conformal vector field, then the manifold is isometric to the Euclidean sphere. We have showed that a Riemannian manifold satisfying gradient $\rho$-Einstein soliton with convex Einstein potential possesses non-negative scalar curvature. We have also deduced a sufficient condition for a Riemannian manifold to be compact which satisfies almost $\eta$-Ricci soliton (see Theorem 2).
\end{abstract}
\noindent\footnotetext{ $^*$ Corresponding author.\\
$\mathbf{2010}$\hspace{5pt}Mathematics\; Subject\; Classification: 53C15, 53C21, 53C44, 58E20, 58J05. }
\maketitle
\section{\textbf{Introduction}}
In 1982, Hamilton \cite{HAM82} introduced the notion of Ricci flow in a Riemannian manifold $(M,g_0)$ to find the various geometric and topological structures of Riemannian manifolds. The Ricci flow is defined by an evolution equation for metrics on $(M,g_0)$:
$$\frac{\partial}{\partial_t}g(t)=-2Ric,\quad g(0)=g_0.$$
A Ricci soliton on a Riemannian manifold $(M,g)$ is a generalization of Einstein metric and is defined as 
\begin{equation}\label{ieq5}
Ric+\frac{1}{2}\mathcal{L}_Xg=\lambda g,
\end{equation}
where $X$ is a smooth vector field on $M$, $\mathcal{L}$ denotes the Lie-derivative operator  and $\lambda\in\mathbb{R}$. It is observed that Ricci solitons are self-similar solutions to the Ricci flow. Ricci soliton is called shrinking,  steady or expanding according as $\lambda>0$, $\lambda=0$ or $\lambda<0$, respectively. The vector field $X$ is called the potential vector field of the Ricci soliton. If $X$ is either Killing or vanishing vector field, then Ricci soliton is called trivial Ricci soliton and (\ref{ieq5}) reduces to an Einstein metric. If $X$ becomes the gradient of a smooth function $f\in C^\infty(M)$, the ring of smooth functions on $M$, then the Ricci soliton is called gradient Ricci soliton and (\ref{ieq5}) reduces to the form
\begin{equation}\label{eq8}
Ric+\nabla^2f=\lambda g,
\end{equation}
where $\nabla^2f$ is the Hessian of $f$.
Perelman \cite{PE02} showed that Ricci soliton on any complete manifold is always a gradient Ricci soliton. If we replace the constant $\lambda$ in (\ref{ieq5}) with a smooth function $\lambda\in C^\infty(M)$, called soliton function, then we say that $(M,g)$ is an almost Ricci soliton, see (  \cite{BR11}, \cite{BR13}, \cite{PRRS11}).\\
\indent Almost gradient Ricci soliton motivated Catino \cite{CAT12} to introduce a new class of Riemannian metrics which are natural generalization of Einstein metrics. In particular, a Riemannian manifold $(M^n,g)$, $n\geq 2$, is called a generalized quasi-Einstein manifold if there are smooth functions $f,\lambda$ and $\mu$ on $M$ such that 
$$Ric+\nabla^2f=\lambda g+\mu df\otimes df.$$
Cho and Kimura \cite{CK09} further generalized the notion of Ricci soliton and developed the concept of $\eta$-Ricci soliton. If a Riemannian manifold $M$ satisfies
$$Ric+\frac{1}{2}\mathcal{L}_Xg=\lambda g+\mu \eta\otimes\eta,$$
for some constant $\lambda$ and $\mu$, then $M$ is said to admit $\eta$-Ricci soliton with soliton vector field $X$.
 A further generalization is the notion of almost $\eta$-Ricci soliton defined by Blaga \cite{BLA}.  
 \begin{definition}\cite{BLA}
 A complete Riemannian manifold $(M,g)$ is said to satisfy almost $\eta$-Ricci soliton if there exists a smooth vector field $X\in\chi(M)$, the algebra of smooth vector fields on $M$, such that
 \begin{equation}\label{eq1}
 Ric+\frac{1}{2}\mathcal{L}_Xg=\lambda g+\mu \eta\otimes\eta,
 \end{equation}
 where $\lambda$ and $\mu$ are smooth functions on $M$ and $\eta$ is an 1-form on $M$. 
 \end{definition}
 \noindent  If $X$ is the gradient of $f\in C^\infty(M)$, then $(M,g)$ is called a gradient almost $\eta$-Ricci soliton. Hence (\ref{eq1}) reduces to the form
\begin{equation}\label{eta1}
Ric+\nabla^2f=\lambda g+\mu\eta\otimes\eta.
\end{equation}
Instead of Ricci flow, Catino and Mazzieri \cite{CAT16} considered the following gradient flow
\begin{equation}
\frac{\partial}{\partial_t}g(t)=-2(Ric-\frac{1}{2}Rg),
\end{equation} 
and introduced the concept of gradient Einstein soliton in a Riemannian manifold, where $R$ is the scalar curvature of the manifold.
\begin{definition}\cite{CAT16}
A Riemannian manifold $(M,g)$ of dimension $n$ is said to be the gradient Einstein Ricci soliton if 
$$Ric-\frac{1}{2}Rg+\nabla^2f=\lambda g,$$
for some function $f\in C^\infty(M)$ and some constant $\lambda\in \mathbb{R}$.
\end{definition}
A more general type gradient Einstein soliton has been deduced by considering the following Ricci-Bourguignon flows \cite{CCDMM17}:
$$ \frac{\partial}{\partial_t}g(t)=-2(Ric-\rho Rg),$$
where $\rho$ is a real non-zero constant.
\begin{definition}\cite{CAT16}
A Riemannian manifold $(M,g)$ of dimension $n \geq 3$ is said to be the gradient $\rho$-Einstein Ricci soliton if 
$$Ric+\nabla^2f=\lambda g+\rho Rg,\quad \rho\in \mathbb{R},\ \rho\neq 0,$$
for some function $f\in C^\infty(M)$ and some constant $\lambda\in \mathbb{R}$. The function $f$ is called Einstein potential. The gradient $\rho$-Einstein soliton is called expanding if $\lambda<0$, steady if $\lambda=0$ and shrinking if $\lambda>0$.
\end{definition}
\indent The paper is arranged as follows: Section 2 discusses some basic concepts of Riemannian
manifold and some definitions, which are needed for the rest of the paper. Section 3
deals with the study of almost $\eta$-Ricci soliton in a complete Riemannian manifold and it is shown that in a compact manifold the potential of such soliton turns into the Hodge-de Rham potential, upto a constant. In this section we have also deduced a sufficient condition for a Riemannian manifold admitting almost $\eta$-Ricci soliton to be compact. In the last section as the main result of the paper we have proved that a compact Riemannian manifold satisfying a gradient $\rho$-Einstein soliton with gradient of Einstein potential as a conformal vector field, is isometric to the Euclidean sphere. We have also studied some properties of gradient $\rho$-Einstein soliton in a complete Riemannian manifold. Among others it is proved that if $(M,g)$ is a compact gradient $\rho$-Einstein soliton with $\rho$ as non-positive real number and gradient of the Einstein potential is a conformal vector field, then such soliton can never be expanding.
\section{\textbf{Preliminaries}}
Throughout this paper by $M$ we mean a complete Riemannian manifold of dimension $n$ endowed with some positive definite metric $g$ unless otherwise stated. In this section we have discussed some rudimentary facts of $M$ (for reference see \cite{PE06}). The tangent space at the point $p\in M$ is denoted by $T_pM$. The geodesic with initial point $p$ and final point $q$ is denoted by $\gamma_{pq}$. A smooth section of the thangent bundle $TM$ is called smooth vector field. The gradient of a smooth function $u:M\rightarrow\mathbb{R}$ at the point $p\in M$ is defined by $\nabla u(p)=g^{ij}\frac{\partial u}{\partial x^j}\frac{\partial}{\partial x^i}\mid_p.$ It is the unique vector field such that any smooth vector field $X$ in $M$ satisfies $g( \nabla u,X)=X(u)$. The Hessian $Hess(u)$ is the symmetric $(0,2)$-tensor field and is defined by $\nabla^2u(X,Y)=Hess(u)(X,Y)=g(\nabla_X\nabla u,Y)$ for all smooth vector fields $X,Y$ of $M$. In local coordinates this can be written as
$$(\nabla^2u)_{ij}=\partial_{ij}u-\Gamma^k_{ij}\partial_ku,$$
where $\Gamma^k_{ij}$ is the Christoffel symbol of $g$. For any vector field $X\in\chi(M)$ and a covariant tensor field $\omega$ of order $r$ on $M$, the Lie derivative of $\omega$ with respect to $X$ is defined by
$$(\mathcal{L}_X\omega)(X_1,\dots,X_r)=X(\omega(X_1,\dots,X_r))-\sum_{i=1}^{r}\omega(X_1,\dots,[X,X_i],\dots,X_n),$$
where $ X_i\in \chi(M)\text{ for }i=1,\dots,r.$ In particular, when $\omega=g$, then 
$$(\mathcal{L}_Xg)(Y,Z)=g(\nabla_YX,Z)+g(Y,\nabla_ZX) \text{ for }Y,Z\in \chi(M).$$
Given a vector field $X$, the divergence of $X$ is defined by
$$div(X)=\frac{1}{\sqrt{g}}\frac{\partial }{\partial x^j}\sqrt{g}X^j,$$
where $g=det(g_{ij})$ and $X=X^j\frac{\partial}{\partial x^j}$. The Laplacian of $u$ is defined by $\Delta u=div(\nabla u)$. 
\begin{definition}\cite{YAU75}
A $C^2$-function $u:M\rightarrow\mathbb{R}$ is said to be harmonic if $\Delta u=0$. The function $u$ is called subharmonic (resp. superharmonic) if $\Delta \geq 0$ (resp. $\Delta u\leq 0$), where $\Delta$ is the Laplacian operator in $M$.
\end{definition}
\begin{definition}\cite{UDR94}
 A function $u:M\rightarrow\mathbb{R}$ is called convex if the following inequality holds
\begin{equation*}
u\circ\gamma(t)\leq (1-t)u\circ\gamma(0)+tu\circ\gamma(1)\quad \forall t\in [0,1],
\end{equation*}
and for any geodesic $\gamma:[0,1]\rightarrow M$. And in case of $u$ is differentiable, then $u$ is convex if and only if $u$ satisfies
\begin{equation*}
g(\nabla u,X)_x\leq u(exp_x\nabla u)-u(x), \ \forall X\in T_xM.
\end{equation*}
\end{definition}

\section{\textbf{Some results of almost $\eta$-Ricci soliton in a compact Riemannian manifold}}
We consider $M$ as a compact orientable Riemannian manifold and $X\in\chi(M)$. Then Hodge-de Rham decomposition theorem \cite{ABR11} implies that $X$ can be expressed as 
$$X=\nabla h+Y,$$
where $h\in C^\infty(M)$ and $div(Y)=0$. The function $h$ is called the Hodge-de Rham potential \cite{BR11}.
\begin{theorem}
Let $(M,g, X,\lambda)$ be a compact gradient almost $\eta$-Ricci soliton. If $M$ is also
a gradient almost $\eta$-Ricci soliton with potential function $f$ then, up to a constant, $f$ equals to the Hodge-de Rham potential.
\end{theorem}
\begin{proof}
Since $(M,g,X,\lambda)$ is a compact almost $\eta$-Ricci soliton, so taking trace of (\ref{eq1}), we get
$$R+div(X)=\lambda n+tr(\mu\eta\otimes\eta).$$
Now Hodge-de Rham decomposition implies that $div(X)=\Delta h$, hence from the above equation, we obtain
$$R=\lambda n-\Delta h+tr(\mu\eta\otimes\eta).$$
Again since $M$ is gradient almost $\eta$-Ricci soliton with Perelman potential $f$, hence taking trace of (\ref{eta1}), we have
$$R=\lambda n-\Delta f+tr(\mu\eta\otimes\eta).$$
equating the last two equations, we get $\Delta(f-h)=0$. Hence $f-h$ is a harmonic function in $M$, but $M$ is compact. Hence $f=h+c$, for some constant $c$.
\end{proof}
\begin{theorem}
Let $(M,g)$ be a complete Riemannian manifold satisfying
\begin{equation}\label{eq2}
Ric+\frac{1}{2}\mathcal{L}_g\geq \lambda g+\mu\eta\otimes\eta,
\end{equation}
where $X$ is a smooth vector field, $\mu$ and $\lambda$ are smooth functions and $\eta$ is an 1-form. Then $M$ is compact if $\|X\|$ is bounded and one of the following conditions holds:\\
(i) $\lambda\geq 0$ and $\mu>c>0$,\\
(ii) $\lambda> c>0$ and $\mu\geq 0,$ \\for some constant $c>0$.
\end{theorem}
\begin{proof}
Let $p\in M$ be a fixed point and $\gamma:(0,\infty]\rightarrow M$ be a geodesic ray such that $\gamma(0)=p$. Then along $\gamma$ we calculate
$$\mathcal{L}_Xg(\gamma',\gamma')=2g(\nabla_{\gamma'}X,\gamma')=2\frac{d}{dt}[g(X,\gamma')].$$
Now from (\ref{eq2}) and above equation, we have
\begin{eqnarray*}
\int_{0}^{T}Ric(\gamma',\gamma')dt&\geq & \int_{0}^{T}\lambda(\gamma(t))g(\gamma',\gamma')dt-\int_{0}^{T}\frac{d}{dt}[g(X,\gamma')]dt+\int_{0}^{T}\mu(\gamma(t))(\eta\otimes\eta)(\gamma',\gamma')dt\\
&=& \int_{0}^{T}\lambda(\gamma(t))dt+g(X_p,\gamma'(0))-g(X_{\gamma(T)},\gamma'(T))+\int_{0}^{T}\mu(\gamma(t))\eta^2(\gamma')dt\\
&\geq & \int_{0}^{T}\lambda(\gamma(t))dt+g(X_p,\gamma'(0))-\|X_{\gamma(T)}\|+\int_{0}^{T}\mu(\gamma(t))\eta^2(\gamma')dt.
\end{eqnarray*}
The last inequality follows from Cauchy-Schwarz inequality. If any one of the conditions $(i)$ and $(ii)$ holds, then above inequality implies that
$$\int_{0}^{\infty}Ric(\gamma',\gamma')dt=\infty.$$
Hence Ambrose's compactness theorem \cite{AM57} implies that $M$ is compact.
\end{proof}
\section{\textbf{Gradient $\rho$-Einstein soliton in a compact Riemannian manifold}}
Throughout this section $M$ is a complete Riemannian manifold with dimension $n\geq 2$.
\begin{theorem}\cite{YAN70}\label{th1}
Suppose $(M,g)$ is a compact Riemannian manifold with constant scalar curvature and $M$ admits a non-trivial conformal vector field $X$. If $\mathcal{L}_XRic=\alpha g$ for some $\alpha\in C^\infty(M)$, then $M$ is isometric to the Euclidean sphere $\mathbb{S}^n$.
\end{theorem}

Let $(M,g)$ be a gradient $\rho$-Einstein soliton. Then
$$Ric+\nabla^2f=\rho Rg+\lambda g.$$ 
If $\nabla f$ is conformal vector field, then $\nabla^2f=\psi g$, for some $\psi\in C^\infty(M)$. Hence above equation reduces to the form
\begin{equation}\label{eq6}
Ric=(\rho R+\lambda-\psi)g.
\end{equation}
Hence Ricci curvature depends only on the points of $M$. Then it follows from Schur's lemma that $R$ is constant.
Again by taking $X=\nabla f$, we have
$$\mathcal{L}_{X}Ric=(\rho R+\lambda-\psi)\mathcal{L}_{X}g=(\rho R+\lambda-\psi)\psi g.$$
Hence from Theorem \ref{th1} we can state the main theorem of the paper:
\begin{theorem}\label{th3}
Let $(M,g)$ be a compact gradient $\rho$-Einstein soliton with Einstein potential $f$. If $\nabla f$ is a non-trivial conformal vector field, then $M$ is isometric to the Euclidean sphere $\mathbb{S}^n$.
\end{theorem}
\begin{theorem}\cite{YAN70}\label{th2}
If $M$ is compact with constant scalar curvature and admits a non-trivial conformal vector field $X$: $\mathcal{L}_Xg=2\psi g$, $\psi\neq 0$, then 
$$\int_M\psi dV=0.$$
\end{theorem}
\noindent Taking the trace in (\ref{eq6}), we get
$$R=n(\rho R+\lambda-\psi),$$
which implies that
$$\int_M(1-n\rho)R=n\int_M(\lambda-\psi).$$
If $X$ is conformal vector field and $M$ is of constant scalar curvature, then applying Theorem \ref{th2} we get
\begin{equation}\label{eq5}
 R\int_M(1-n\rho)=n\int_M\lambda.
\end{equation}
Now if $\lambda<0$, then the above equation becomes
$$ R\int_M(1-n\rho)<0.$$
If $M$ is compact, then Theorem \ref{th3} implies that $M$ is isometric to $\mathbb{S}^n$. Since isometry preserves scalar curvature so $R>0$. Hence the above equation implies that
\begin{equation}
Vol(M)<n\int_M\rho.
\end{equation}
Hence we can state the following:
\begin{theorem}
Let $(M,g)$ be a compact gradient $\rho$-Einstein soliton with Einstein potential $f$ and $\rho\leq 0$. If $\nabla f$ is conformal vector field then $M$ is shrinking or steady gradient $\rho$-Einstein soliton.
\end{theorem}
\begin{lemma}\cite{CAT16}
Let $(M,g)$ be gradient $\rho$-Einstein Ricci soliton with Einstein potential $f$. Then we have
\begin{equation}\label{eq3}
\Delta f=-(1-n\rho)R+n\lambda.
\end{equation}
\end{lemma}
\begin{proposition}
Suppose $(M,g)$ is an expanding or steady gradient $\rho$-Einstein Ricci soliton with Einstein potential $f$ and $n\rho>1$. If $f$ is a convex function, then $M$ has non-negative scalar curvature.
\end{proposition}
\begin{proof}
The convexity of $f$ implies that $f$ is subharmonic \cite{GW71}, i.e., $\Delta f\geq 0$. Hence (\ref{eq3}) implies that
$$(1-n\rho)R-n\lambda\leq 0.$$
Now take $1-n\rho=-h$, where $h>0$ is a real constant. Then we get 
\begin{equation}\label{eq4}
R\geq -\frac{n\lambda}{h}.
\end{equation}
 Since $M$ is expanding or steady, so $\lambda\leq 0$. Hence we can conclude from (\ref{eq4}) that $R\geq 0$.
\end{proof}
The following can be easily derived from (\ref{eq3}):
\begin{proposition}
Suppose $(M,g)$ is a steady gradient $\rho$-Einstein Ricci soliton with Einstein potential $f$ and $n\rho>1$. If $f$ is a harmonic function, then the scalar curvature of $M$ vanishes.
\end{proposition}
Integrating (\ref{eq5}) on $M$, we get
$$R(1-n\rho)Vol(M)=n\lambda Vol(M),$$
which yields
$$R=\frac{n\lambda}{1-n\rho}.$$
If $R>0$, then $n\lambda>1-n\rho$, i.e., $\rho>\frac{1}{n}(1-n\lambda).$ Hence Theorem \ref{th3} implies that
\begin{proposition}
Let $(M,g)$ be a compact gradient $\rho$-Einstein soliton with Einstein potential $f$. If $\nabla f$ is a non-trivial conformal vector field, then $\rho$ satisfies
$$\rho>\frac{1}{n}(1-n\lambda).$$
\end{proposition}
\section*{\textbf{Acknowledgment}}
 The second author greatly acknowledges to The University Grants Commission, Government of India for the award of Junior Research Fellowship.

$\bigskip $

$^{1,2}$The University of Burdwan, 

Department of Mathematics,

 Golapbag, Burdwan-713104,
 
 West Bengal, India.
 
$^1$E-mail:aask2003@yahoo.co.in, aashaikh@math.buruniv.ac.in

$^2$E-mail:chan.alge@gmail.com

\end{document}